\documentclass[11pt]{article}
\usepackage{url}
\usepackage{amsfonts, mathrsfs, selinput}
\usepackage{amssymb, amscd}
\usepackage{amsmath}
\usepackage{amsthm,mathabx}
\usepackage[lofdepth,lotdepth]{subfig}
\usepackage{setspace,amssymb,verbatim, enumitem}
\usepackage[margin=1in]{geometry}

\usepackage{lineno}

\SelectInputMappings{
adieresis={ä},
ntilde={ñ}
}
\usepackage{prettyref}
\newrefformat{cond}{condition (\ref{#1})}
\newrefformat{conj}{Conjecture \ref{#1}}
\usepackage{graphicx}
\usepackage{rotating}
\usepackage{psfrag}
\usepackage{afterpage}
\usepackage{multirow}
\usepackage[lofdepth,lotdepth]{subfig}
\usepackage{tikz}

\theoremstyle{plain}
\newtheorem{theorem}{Theorem}[section]
\newtheorem{proposition}[theorem]{Proposition}
\newtheorem{prop}[theorem]{Proposition}
\newtheorem*{theorem*}{Theorem}

\newtheorem*{corollary*}{Corollary}

\newtheorem{cor}[theorem]{Corollary}

\newtheorem*{lemma*}{Lemma}
\newtheorem{lemma}[theorem]{Lemma}
\newtheorem{lem}[theorem]{Lemma}

\theoremstyle{remark}

\newtheorem*{case*}{Case}

\theoremstyle{definition}

\newtheorem{thml}{Theorem}

\usepackage{prettyref}
\newrefformat{sec}{Section \ref{#1}}
\newrefformat{chap}{Chapter \ref{#1}}
\newrefformat{prop}{Proposition \ref{#1}}
\newrefformat{prob}{Problem \ref{#1}}
\newrefformat{rem}{Remark \ref{#1}}
\newrefformat{cor}{Corollary \ref{#1}}
\newrefformat{cond}{Condition \ref{#1}}
\newrefformat{def}{Definition \ref{#1}}
\newrefformat{ex}{Example \ref{#1}}
\newrefformat{lem}{Lemma \ref{#1}}
\newrefformat{eq}{Equation \eqref{#1}}
\newrefformat{con}{condition (\ref{#1})}
\newrefformat{conj}{Conjecture \ref{#1}}
\newrefformat{fig}{Figure \ref{#1}}
\newrefformat{tab}{Table \ref{#1}}
\newrefformat{app}{Appendix \ref{#1}}

\newcommand{\tw}[1]{{}^#1\!}

\newcommand{\aut}{\mathrm{Aut}}

\newcommand{\F}{\mathbb{F}}

\newcommand{\irr}{\mathrm{Irr}}

\newcommand{\la}{\lambda}

\newcommand{\wt}[1]{\widetilde{#1}}

\newcommand{\bg}[1]{\textbf{#1}}

\def\irr#1{{\Irr}(#1)}
\def\irrp#1{{\Irr}_{p'}(#1)}

\def\Oh#1#2{{\bf O}^{#1}(#2)}

\def\aut#1{{\rm Aut}(#1)}

\let\phi=\varphi

\newcommand{\cH}{{\mathcal H}}

\newcommand{\Irr}{\operatorname{Irr}}
\newcommand{\GL}{{\operatorname{GL}}}
\newcommand{\GU}{{\operatorname{GU}}}
\newcommand{\SL}{{\operatorname{SL}}}
\newcommand{\PGL}{{\operatorname{PGL}}}
\newcommand{\PSL}{{\operatorname{PSL}}}
\newcommand{\PSp}{{\operatorname{PSp}}}
\newcommand{\Sp}{{\operatorname{Sp}}}
\newcommand{\CSp}{{\operatorname{CSp}}}
\newcommand{\PSU}{{\operatorname{PSU}}}

\newcommand{\6}{^}
\let\la=\lambda

\newcommand{\normal}{\lhd}

\title{Groups with few $p'$-character degrees }
\author{
Eugenio Giannelli \\ \small\textit{Dipartimento di Matematica e Informatica}\\ \small\textit{U. Dini,
Viale Morgagni 67/a}\\ \small\textit{ Firenze, Italy} \\ \small\textit{eugenio.giannelli@unifi.it}
\and
Noelia Rizo \\ \small\textit{Departament de Matem{\`a}tiques}\\ \small\textit{Universitat de Val{\`e}ncia}\\ \small\textit{46100 Burjassot, Val{\`e}ncia, Spain} \\ \small\textit{noelia.rizo@uv.es}
\and
A. A. Schaeffer Fry \\ \small\textit{Department of Mathematical and Computer Sciences}\\\small\textit{Metropolitan State University of Denver}\\ \small\textit{Denver, CO 80217, USA}\\ \small\textit{aschaef6@msudenver.edu}}

\date{}
\begin{document}
\maketitle

\begin{abstract}
We prove a variation of Thompson's Theorem. 
Namely, if the first column of the character table of a finite group $G$ contains only two distinct values
not divisible by a given prime number $p>3$, then $\Oh{pp'pp'}{G}=1$. This is done by using the classification of finite simple groups. 

\vspace{0.25cm}

\noindent \textit{Mathematics Classification Number:} 20C15, 20C33


\end{abstract}

\section{Introduction}

Many problems in the character theory of finite groups deal with character degrees and prime numbers. For instance, the It\^o-Michler theorem (\cite{ito},\cite{michler}) asserts that a prime $p$ does not divide $\chi(1)$ for any $\chi\in{\rm Irr}(G)$ if and only if the group $G$ has a normal abelian Sylow $p$-subgroup. As usual, we denote by ${\rm Irr}(G)$ the set of complex irreducible characters of $G$ and by ${\rm Irr}_{p'}(G)$ the subset of ${\rm Irr}(G)$ consisting of irreducible characters of degree coprime to $p$. If we write ${\rm cd}(G)$ for the set of character degrees of $G$ and ${\rm cd}_{p'}(G)$ for the subset of ${\rm cd}(G)$ consisting of those character degrees not divisible by $p$, the It\^o-Michler theorem deals with the situation when ${\rm cd}_{p'}(G)={\rm cd}(G)$. 
At the opposite end of the spectrum, we have a theorem of Thompson (\cite{thompson}) showing that if ${\rm cd}_{p'}(G)=\{1\}$, then $G$ has a normal $p$-complement. 
%A kind of dual situation is  
In \cite{berk}, Berkovich  showed that in this situation, $G$ is solvable, using the Classification of Finite Simple Groups. 

In \cite{GSV}, a variation on Thompson's theorem for two primes is studied.  Namely, it is shown there that if $G$ has only one character with degree coprime to two different primes, then $G$ is trivial.  In this paper, we offer another variation on Thompson's theorem by describing finite groups $G$ such that $|{\rm cd}_{p'}(G)|=2$.  Thompson's theorem may equivalently be viewed in terms of $\Oh{pp'}{G}$, where $\Oh{pp'}{G}=\Oh{p'}{\Oh{p}{G}}$.  Namely, Thompson's theorem says that if ${\rm cd}_{p'}(G)=1$ then $\Oh{pp'}{G}=1$.  Similarly, we will denote by $\Oh{pp'pp'}{G}$ the group $\Oh{pp'}{\Oh{pp'}{G}}$ and prove a corresponding statement for the case $|{\rm cd}_{p'}(G)|=2$.

%Equivalently, Thompson's theorem says that if ${\rm cd}_{p'}(G)=1$ then $\Oh{pp'}{G}=1$, where $\Oh{pp'}{G}=\Oh{p'}{\Oh{p}{G}}$.  Similarly, we will denote by $\Oh{pp'pp'}{G}$ the group $\Oh{pp'}{\Oh{pp'}{G}}$.  In this paper, we describe finite groups $G$ such that $|{\rm cd}_{p'}(G)|=2$.  

\begin{thml} Let $G$ be a finite group and let  $p>3$ be a prime. Suppose that $|{\rm cd}_{p'}(G)|=2$. Then $G$ is solvable and $\Oh{pp'pp'}{G}=1$.
\label{thmA}
\end{thml}

We remark that Theorem \ref{thmA} answers Problem 5.3 of \cite{navitomich}, which suggested that groups satisfying $|{\rm cd}_{p'}(G)|=2$ must be $p$-solvable. It is already said in \cite{navitomich} that there are many examples of non-solvable groups satisfying $|{\rm cd}_{2'}(G)|=2$ (the symmetric group on 5 letters is the smallest one). For $p=3$, we have that if $G$ is the automorphism group of ${\rm PSL}_2(27)$, then $|{\rm cd}_{3'}(G)|=2$, for instance.

We also remark that, for a given prime $p$, we can find examples of non-solvable groups $G$ satisfying $|{\rm cd}_{p'}(G)|=3$. For instance, the alternating group on 5 letters, $A_5$, satisfies that $|{\rm cd}_{p'}(A_5)|=3$ for $p=2,3,5$. For $p>5$, we have that $|{\rm cd}_{p'}({\rm PGL}_2(p))|=3$.

Further, unlike many statements on character degrees, Theorem A does not immediately seem to correspond to a \textit{dual} statement for conjugacy classes.  For instance, we observe that $A_5$ has exactly two conjugacy classes of size coprime to $5$.

%Sometimes statements on character degrees corresponds to \textit{dual} statements concerning conjugacy classes. This is not the case for Theorem A. For instance, we observe that $A_5$ has exactly two conjugacy classes of size coprime to $5$. 

\medskip

 %The next natural question asks what we can say about the structure of $G$ when $|{\rm cd}_{p'}(G)|=2$. (This is Problem 5.3  of \cite{navitomich}). In \cite{navitomich}, it is suggested that for $p>3$, the groups satisfying this condition are $p$-solvable. Here we prove more. 

As in the case of the result of Berkovich, our proof of Theorem \ref{thmA} uses the Classification of Finite Simple Groups. In particular, we need the following result on simple groups. %As usual, we denote by ${\rm Irr}_{p'}(G)$ the set of irreducible characters of $G$ with degree not divisible by $p$. 

\begin{thml} Let $S$ be a non-abelian simple group and  let $p>3$ be a prime. Then there exist nonlinear $\alpha, \beta \in{\rm Irr}_{p'}(S)$ such that $\alpha$ extends to ${\rm Aut}(S)$, $\beta(1)\nmid\alpha(1)$, and $\beta$ is $P$-invariant for every $p$-subgroup $P\leq{\rm Aut}(S)$.
\label{thmB}
\end{thml}

We obtain Theorem \ref{thmB} as a corollary to the following stronger statement, after dealing with the exceptions separately.

\begin{thml}\label{thmC}
Let $S$ be a non-abelian simple group and let $p>3$ be a prime dividing $|S|$.  Assume that $S$ is not one of $A_5, A_6$ for $p=5$ or one of $\PSL_2(q), \PSL^\epsilon_3(q),$ $\PSp_4(q)$, or $\tw{2}B_2(q)$. Then there exist two nontrivial characters $\alpha, \beta\in\irrp S$ such that $\alpha(1)\neq \beta(1)$ and both $\alpha$ and $\beta$ extend to $\aut S$.
\end{thml}
This article is structured as follows. In Section \ref{sec:red}, we prove Theorem \ref{thmA} assuming that Theorem \ref{thmB} is true and show that Theorem \ref{thmC} (and hence Theorem \ref{thmB}) holds for sporadic groups.  In Section \ref{sec:alt}, we prove Theorem \ref{thmB} for the alternating groups. Finally, in Section \ref{sec:lietype}, we prove Theorem \ref{thmB} for simple groups of Lie type and we conclude by using the Classification of Finite Simple Groups.

\bigskip

\textbf{Acknowledgement} The authors would like to thank Gabriel Navarro for many useful conversations on this topic.  They would also like to thank the Mathematisches Forschungsinstitut Oberwolfach and the organizers of the 2019 MFO workshop ``Representations of Finite Groups", where part of this work was completed.  The second author is supported by a Fellowship FPU of Ministerio de Educación, Cultura y Deporte. The third author is partially supported by a grant from the National Science Foundation (Award No. DMS-1801156).

\section{Reduction to simple groups}\label{sec:red}

In this section we assume Theorem \ref{thmB} and we prove Theorem \ref{thmA}. We will need the following result, which is essentially Lemma 4.1 of \cite{39a63c1a9549405fbc1f1aa638980ab9}. We provide a proof for the reader's convenience.

\begin{lemma} Let $Q$ be a finite group acting on a finite group $N$ and suppose that $N=S_1\times\cdots\times S_t$, where the $S_i's$ are transitively permuted by $Q$. Let $Q_1$ be the stabilizer of $S_1$ in $Q$ and let $T=\{a_1,\cdots, a_t\}$ be a transversal for the right cosets of $Q_1$ in $Q$ with $S_i=S_1^{a_i}$. Let $\psi\in{\rm Irr}(S_1)$ be $Q_1$-invariant. Then $$\gamma=\psi^{a_1}\times \psi^{a_2}\times \cdots\times\psi^{a_t}\in{\rm Irr}(N)$$ is $Q$-invariant.
\label{lemaq}
\end{lemma}

\begin{proof}

First, we have that $Q$ acts transitively on $T$, and we use the notation $a_i \cdot q$ to indicate the
unique element of $T$ such that $Q_1(a_iq)=Q_1(a_i\cdot q)$ for $a_i\in T$ and $q\in Q$. Now notice that $Q$ acts on $N$ as follows: if $q \in Q$ and $S_i^q=S_j$ (that is, if $a_i\cdot q=a_j$), then $(x_1,...,x_t)^q =(y_1,\ldots,y_t)$, where $y_j=x_i^q$. Let $q\in Q$, and let $\sigma\in S_t$ be the permutation defined by $a_i\cdot q= a_{\sigma(i)}$, so $x_i^q=y_{\sigma(i)}$. Then

\begin{align*}
\gamma^{q^{-1}}(x_1,\ldots, x_t)&=(\psi^{a_1}\times \psi^{a_2}\times \cdots\times\psi^{a_t})^{q^{-1}}(x_1,\ldots, x_t)\\
&=\psi^{a_1}(y_1)\cdots\psi^{a_t}(y_t)\\
&=\psi^{a_1}(x_{\sigma^{-1}(1)}^q)\cdots\psi^{a_t}(x^q_{\sigma^{-1}(t)})\\
&=\psi^{a_1q^{-1}}(x_{\sigma^{-1}(1)})\cdots\psi^{a_tq^{-1}}(x_{\sigma^{-1}(t)})\\
&=\psi^{a_1\cdot q^{-1}}(x_{\sigma^{-1}(1)})\cdots\psi^{a_t\cdot q^{-1}}(x_{\sigma^{-1}(t)})\\
&=\psi^{a_{\sigma^{-1}(1)}}(x_{\sigma^{-1}(1)})\cdots\psi^{a_{\sigma^{-1}(t)}}(x_{\sigma^{-1}(t)})\\
&=\psi^{a_1}(x_1)\cdots \psi^{a_t}(x_t)\\
&=\gamma(x_1,\ldots, x_t).
\end{align*}
\end{proof}

\begin{theorem} Let $G$ be a finite group and let  $p>3$ be a prime. Suppose that $|{\rm cd}_{p'}(G)|=2$. Then $G$ is solvable.
\end{theorem}

\begin{proof}

We argue by induction on $|G|$. 
Write ${\rm cd}_{p'}(G)=\{1,m\}$. Let $N$ be a minimal normal subgroup of $G$. Then either $N$ is abelian or $N$ is semisimple. It is clear that $|{\rm cd}_{p'}(G/N)|\leq |{\rm cd}_{p'}(G)|$. Hence by induction and Proposition 9 of \cite{berk} we have that $G/N$ is solvable. 

Suppose that $N=S_1\times S_2\times\cdots\times S_t$ where $S_i\cong S$, a non-abelian simple group. Write $H={\rm\textbf{N}}_G(S_1)$ and $S_i=S_1^{x_i}$, where $G=\bigcup_{i=i}^t Hx_i$ is a disjoint union. By Theorem \ref{thmB}, there exist $\alpha\in{\rm Irr}_{p'}(S_1)$ with $\alpha(1)\neq 1$, extending to ${\rm Aut}(S_1)$,  and $\beta\in{\rm Irr}_{p'}(S_1)$ $P$-invariant for every $p$-subgroup $P\leq{\rm Aut}(S_1)$, with $\beta(1)\nmid \alpha(1)$. 

Now let $\theta=\alpha^{x_1}\times\cdots\times \alpha^{x_t}\in{\rm Irr}(N)$. By Lemma 5 of \cite{BCLP} we have that $\theta$ extends to $G$. Let $\tilde{\theta}\in{\rm Irr}(G)$ extending $\theta$. Then $\tilde{\theta}(1)=\theta(1)=\alpha(1)^t$ is not divisible by $p$ and hence, by hypothesis, $\tilde{\theta}(1)=m$.

Let $Q\in{\rm Syl}_p(G)$ and write $\{T_1,\ldots, T_r\}$ for a set of representatives of the action of $Q$ on $\{S_1,\ldots, S_t\}$, with $T_1=S_1$. Write $\mathcal{O}(T_i)$ for the orbit of $T_i$ under the action of $Q$. Rearrange $S_1, S_2,\ldots, S_t$ in such a way that $\mathcal{O}(T_1)=\{S_{1},S_2,\cdots,S_{l_1}\}$, $\mathcal{O}(T_2)=\{S_{l_1+1},\ldots, S_{l_2}\}$, etc. Notice that $l_1+l_2+\cdots l_r=t$. Write $N_1=S_{1}\times S_{2}\times\cdots\times S_{l_1}$, $N_2=S_{l_1+1}\times\cdots\times S_{l_2}$, etc. Notice that $Q$ normalizes $N_i$ and $N=N_1\times N_2\times\cdots\times N_r$. 

Now, let $U=\{q_1,q_2,\ldots, q_{l_1}\}$ be a transversal for the right cosets of $Q_1=Q\cap {\rm \textbf{N}}_G(S_1)$ in $Q$ such that $S_{j}=S_1^{q_{j}}$  for $j=1,2,\ldots, l_1$,  and define $\gamma_1\in{\rm Irr}(N_1)$ as follows:

$$\gamma_1=\beta^{q_{1}}\times \beta^{q_{2}}\times\cdots\times \beta^{q_{l_1}}\in{\rm Irr}(N_1).$$ By Lemma \ref{lemaq} we have that $\gamma_1$ is $Q$-invariant. If $T_i=S_1^{x_k}$ proceed analogously with $\beta^{x_k}$ to define $\gamma_i\in{\rm Irr}(N_i)$ (notice that $\beta^{x_k}$ is $Q_k$-invariant, where $Q_k={\rm\textbf{N}}_G(S_1^{x_k})\cap Q$). By Lemma \ref{lemaq}, we have that $\gamma_i$ is $Q$ invariant for all $i=1,2,\ldots, r$.

%Now, for $j=1,2,\ldots, l_1$, write $S_{j}=T_1^{q_{j}}$ for some $q_{j}\in Q$, and define $\gamma_1\in{\rm Irr}(N_1)$ as follows. If $T_1=S_1^{x_k}$, then $\beta^{x_k}\in{\rm Irr}(T_1)$ and hence $(\beta^{x_k})^{q_{j}}\in{\rm Irr}(T_1^{q_{j}})={\rm Irr}(S_{j})$, for $j=1,2,\ldots, l_1$. Now write
%
%$$\gamma_1=(\beta^{x_k})^{q_{1}}\times (\beta^{x_k})^{q_{2}}\times\cdots\times (\beta^{x_k})^{q_{l_1}}\in{\rm Irr}(N_1).$$ By Lemma 4.1 of \cite{39a63c1a9549405fbc1f1aa638980ab9} we have that $\gamma_1$ is $Q$-invariant. Define analogously $\gamma_i\in{\rm Irr}(N_i)$ for $i=1,\ldots, r$, and notice that $\gamma_i$ is $Q$ invariant by Lemma 4.1 of \cite{39a63c1a9549405fbc1f1aa638980ab9}.
%
 Now, let $$\gamma=\gamma_1\times \cdots\times \gamma_{r}\in{\rm Irr}(N).$$ We claim that $\gamma$ is $Q$-invariant. Indeed, let $q\in Q$ and let $n_1n_2\cdots n_r\in N$, with $n_i\in N_i$. Since $Q$ normalizes $N_i$ for all $i$, we have:

$$\gamma^q(n_1n_2\cdots n_r)=\gamma_1^q(n_1)\gamma_2^q(n_2)\cdots\gamma_r^q(n_r)=\gamma_1(n_1)\cdots\gamma_r(n_r).$$

Notice that $\gamma(1)=\beta(1)^t$ and hence, $\gamma\in{\rm Irr}_{p'}(N)$. Then ${\rm gcd}(o(\gamma)\gamma(1),|NQ:N|)=1$ and since $\gamma$ is $Q$-invariant, we have that $\gamma$ extends to $\tilde{\gamma}\in{\rm Irr}(NQ)$. Since $\tilde{\gamma}(1)$ and $|G:NQ|$ are not divisible by $p$, there exists $\chi\in{\rm Irr}(G|\tilde{\gamma})$ of $p'$-degree. But $\chi\in{\rm Irr}(G|\gamma)$, and hence $\gamma(1)\mid \chi(1)$. Since $\gamma(1)\neq 1$, we have that $\chi(1)\neq 1$, and therefore $\chi(1)=m$. Hence $\beta(1)^t=\gamma(1)$ divides $\chi(1)=m=\alpha(1)^t$, and $\beta(1)\mid \alpha(1)$. This contradiction shows that $N$ is abelian and hence $G$ is solvable.
\end{proof}

The following is the second part of Theorem \ref{thmA}.

\begin{theorem} Let $G$ be a finite group and let  $p>3$ be a prime. Suppose that $|{\rm cd}_{p'}(G)|=2$. Then ${\rm \textbf{O}}^{pp'pp'}(G)=1$.
\end{theorem}
\begin{proof} Let $K={\rm\textbf{O}}^p(G)$, $L={\rm\textbf{O}}^{p'}(K)={\rm\textbf{O}}^{pp'}(G)$, $N={\rm\textbf{O}}^p(L)={\rm\textbf{O}}^{pp'p}(G)$ and $W={\rm\textbf{O}}^{p'}(N)={\rm\textbf{O}}^{pp'pp'}(G)$, and write ${\rm cd}_{p'}(G)=\{1,m\}$. Let $W/X$ be a chief factor of $G$. Then $W/X$ is a minimal normal subgroup of $G/X$ and since $G$ is solvable, we have that $W/X$ is abelian. Now, if a prime $q$, different from $p$, divides $|W:X|$, we have that $W/X$ has a normal $q$-complement $H/X$ and $|W:H|$ is not divisible by $p$, a contradiction with the fact that $W={\rm \textbf{O}}^{p'}(N)$. Hence $W/X$ is a $p$-group.

Now, if $X>1$, since ${\rm cd}_{p'}(G/X)\subseteq{\rm cd}_{p'}(G)$, by induction and Thompson's theorem we have that ${\rm \textbf{O}}^{pp'pp'}(G/X)=X$ and hence $W=X$, a contradiction. Therefore we may assume that $X=1$. 

Let $Y$ be a complement of $W$ in $N$. By the Frattini argument, we have that $G=N{\rm\textbf{N}}_G(Y)=W{\rm\textbf{N}}_G(Y)$. Since $W$ is abelian, we have that ${\rm\textbf{C}}_W(Y)\normal G$ and then $W\cap {\rm\textbf{N}}_G(Y)={\rm\textbf{C}}_W(Y)=1$. Hence $W$ is complemented in $G$ and by Problem 6.18 of \cite{Isa76}, every $\lambda\in{\rm Irr}(W)$ extends to $G_\lambda$. %notice that $G_\lambda=W{\rm\textbf{N}}_{G_\lambda}(Y)$.

Let $P$ be a Sylow $p$-subgroup of $G$, let $1_W=\lambda_1,\lambda_2\ldots,\lambda_t$ be representatives of the action of $P$ on ${\rm Irr}(W)$ and let $\mathcal{O}_i$ be the $P$-orbit of $\lambda_i$. Then $$1+\sum_{i=2}^t|\mathcal{O}_i|\lambda_i(1)^2= \sum_{i=1}^t|\mathcal{O}_i|\lambda_i(1)^2=\sum_{\lambda\in{\rm Irr}(W)}\lambda(1)^2=|W|\equiv 0 \quad {\rm mod}\>p.$$  Then there exists $i>1$ such that $|\mathcal{O}_i|\lambda_i(1)^2$ is not divisible by $p$. Since $|\mathcal{O}_i|=|P:P_{\lambda_i}|$, we have that $|\mathcal{O}_i|=1$ and hence there is $1_W\neq\lambda\in{\rm Irr}(W)$ $P$-invariant. Let $T=G_\lambda$ and let $\hat{\lambda}$ be an extension of $\lambda$ to $T$. Then $\hat{\lambda}^G\in{\rm Irr}(G)$ and $\hat{\lambda}^G(1)=|G:T|$. But $P\subseteq T$ and hence $|G:T|=p'$. By hypothesis, $|G:T|=m$.

If $\lambda$ is $N$ invariant, we have that $$\lambda(n^{-1}w^{-1}nw)=\lambda^n(w^{-1})\lambda(w)=\lambda(1)$$ and $[N,W]\subseteq{\rm ker}(\lambda)$. But $[N,W]\normal G$ and $[N,W]\subseteq W$. Since $W$ is a minimal normal subgroup of $G$ we have that $[N,W]=W$ or $[N,W]=1$. If $[N,W]=1$, we have that $[Y,W]\subseteq[N,W]=1$ and hence by Lemma 4.28 of \cite{isaacsfgt} we have that $W=1$ and we are done. Thus we may assume that $[N,W]=W$. Then $W\subseteq{\rm ker}(\lambda)$, which contradicts the fact that $\lambda\neq 1_W$. Hence $N\not\subseteq T$. 

If ${\rm cd}_{p'}(G/N)=\{1\}$, by Thompson's theorem we are done. Then we may assume that ${\rm cd}_{p'}(G/N)=\{1,m\}$. Let $\chi\in{\rm Irr}(G/N)$ with $\chi(1)=m$ and let $\epsilon\in{\rm Irr}(TN/N)$ lying under $\chi$ of $p'$-degree. Since $N\subseteq{\rm ker}(\epsilon)$, we have that $\epsilon_T$ is irreducible and $\rho=\epsilon_T\hat{\lambda}\in{\rm Irr}(T|\lambda)$. Thus $\rho^G\in{\rm Irr}(G|\lambda)$ and $$\rho^G(1)=|G:T|\rho(1)=|G:T|\epsilon(1)$$ is not divisible by $p$. Then $|G:T|\epsilon(1)=m=|G:T|$ and $\epsilon$ is linear. Now, since $\chi$ is an irreducible constituent of $\epsilon^G$ we have that $$m=\chi(1)\leq \epsilon^G(1)=|G:TN|<|G:T|=m.$$ This contradiction shows that $W=1$. %\epsilon_T irreducible by Theorem 1.18 of Gabriel's book, for instance.
\end{proof}

We end this section with the following, which can be seen using the GAP Character Table Library.

\begin{proposition}\label{prop:sporadic}
Let $S$ be a simple sporadic group or the Tits group $\tw{2}F_4(2)'$.  Then Theorem \ref{thmC} holds for $S$.
\end{proposition}

\section{Alternating Groups}\label{sec:alt}
The aim of this section is to prove Theorem B for alternating groups. We achieve this goal by proving the slightly stronger Proposition \ref{cor: An}. 

We begin by recalling some basic facts in the representation theory of symmetric and alternating groups. 
The irreducible characters of the symmetric group $S_n$ are labelled by partitions of $n$. We let $\mathcal{P}(n)$ be the set of partitions of $n$. Given $\lambda=(\lambda_1,\ldots, \lambda_\ell)\in\mathcal{P}(n)$, we denote by $\chi^\lambda$ the corresponding irreducible character of $S_n$. We sometimes use the notation $\lambda\vdash n$ to say that $\lambda\in\mathcal{P}(n)$ and we use the symbol $\lambda\vdash_{p'}n$ to mean that $\chi^\lambda\in\mathrm{Irr}_{p'}(S_n)$, for some prime number $p$. Given $\lambda\in\mathcal{P}(n)$ we let $\lambda'$ be the conjugate partition of $\lambda$. 
 We recall that the restriction to the alternating group $(\chi^\la)_{A_n}$ is irreducible if,
and only if, $\la\neq \la'$ (see \cite[Thm.~2.5.7]{JK}). Moreover, if $\lambda, \mu\in\mathcal{P}(n)$ then $(\chi^\la)_{A_n}=(\chi^\mu)_{A_n}$ if, and only if, $\mu\in\{\lambda, \lambda'\}$. 
%Otherwise $(\chi^\la)_{A_n}=\phi+\phi^{g}$ for
%some $\phi\in\Irr(A_n)$ and $g\in{\sf S}_n\smallsetminus{\sf A}_n$.

%The aim of this section is to find for all $n\in\mathbb{N}$ two partitions $\lambda, \mu\vdash_{p'}n$ such that $1<\chi^\lambda(1)<\chi^\mu(1)$ with $\lambda\neq\lambda'$ and $\mu\neq\mu'$. 

Given $\lambda\in\mathcal{P}(n)$ we denote by $[\lambda]$ its corresponding Young diagram. The node $(i,j)$ of $[\lambda]$ lies in row $i$ and column $j$. As usual, we denote by $h_{(i,j)}(\lambda)$ the hook-length corresponding to $(i,j)$. We let $\mathcal{H}(\lambda)$ be the multiset of hook-lengths in $\lambda$. 
By the hook-length formula \cite[Thm.~2.3.21]{JK} we know that $\chi^\lambda(1)=n!(\prod_{h\in\mathcal{H}(\lambda)}h)^{-1}$. 

Given $e\in\mathbb{N}$ we let $\mathcal{H}^e(\lambda)$ be the multiset of hook-lenghts in $\mathcal{H}(\lambda)$ that are divisible by $e$. 
The $e$-core $C_e(\lambda)$ of $\lambda$ is the partition of $n-e|\mathcal{H}^e(\lambda)|$ obtained from $\lambda$ by successively removing hooks of length $e$. A useful consequence of work of Macdonald \cite{Mac} is stated here (see \cite[Lemma 2.4]{GSV} for a short proof). 
\begin{lem}   \label{lem: last layer}
Let $p$ be a prime and let $n$ be a natural number with $p$-adic expansion
 $n=\sum_{j=0}^ka_jp^j$. Let $\la$ be a partition of $n$. Then $\la\vdash_{p'}n$ if and only if $|\cH^{p^k}(\la)|=a_k$ and $C_{p^k}(\lambda)\vdash_{p'}n-a_kp^k$. 
\end{lem}

We let $\mathcal{L}(n)=\{\lambda\in\mathcal{P}(n)\ |\ \lambda_2\leq 1\}$. Elements of $\mathcal{L}(n)$ are usually called \textit{hook partitions}. We denote by $\mathcal{L}_{p'}(n)$ the set consisting of all those hook partitions of $n$ whose corresponding irreducible character of $S_n$ has degree coprime to $p$. 
In the following proposition we completely describe the elements of $\mathcal{L}_{p'}(n)$ and we give a closed formula for $|\mathcal{L}_{p'}(n)|$. This might be known to experts in the area, but we were not able to find an appropriate reference in the literature. 

%\begin{prop}\label{lem:p'hooks}
%Let $n\in\mathbb{N}$ and let $n=\sum_{j=1}^ka_jp^{n_j}$ be its $p$-adic expansion, where $0\leq n_1<n_2<\cdots <n_k$. Then $$|\mathcal{L}_{p'}(n)|=a_1p^{n_1}\cdot\prod_{j=2}^k(a_j+1).$$
%Moreover, for $\lambda,\mu\in\mathcal{L}(n)$ we have that $\chi^\lambda(1)=\chi^\mu(1)$ if and only if $\lambda=\mu$.
%\end{prop}
\begin{prop}\label{lem:p'hooks}
Let $n\in\mathbb{N}$ and let $n=\sum_{j=1}^ka_jp^{n_j}$ be its $p$-adic expansion, where $0\leq n_1<n_2<\cdots <n_k$ and $1\leq a_j\leq p-1$ for all $j\in\{1,\ldots, k\}$. Then $$|\mathcal{L}_{p'}(n)|=a_1p^{n_1}\cdot\prod_{j=2}^k(a_j+1).$$
Moreover, for $\lambda,\mu\in\mathcal{L}(n)$ we have that $\chi^\lambda(1)=\chi^\mu(1)$ if and only if $\lambda\in\{\mu,\mu'\}$.
\end{prop}
\begin{proof}
We proceed by induction on $k$, the $p$-adic length of $n$. 
Suppose first that $k=1$. If $n_1=0$ then $n=a_1<p$ and therefore $|\mathcal{L}_{p'}(n)|=|\mathcal{L}_(n)|=a_1$. On the other hand, if $n_1>0$ and $\lambda\in\mathcal{L}(n)$ then $h_{(1,1)}(\lambda)=n=a_1p\6{n_1}$. It follows that $|\mathcal{H}\6{p\6{n_1}}(\lambda)|=a_1$ and that $C_{p\6{n_1}}(\lambda)=\emptyset$. Using Lemma \ref{lem: last layer} we conclude that $\lambda\in\mathcal{L}_{p'}(n)$ and hence that $\mathcal{L}(n)=\mathcal{L}_{p'}(n)$. Therefore, we have that $|\mathcal{L}_{p'}(n)|=a_1p\6{n_1}$.
Let us now assume that $k\geq 2$. Let $m:=n-a_kp\6{n_k}$ and let $\gamma\in\mathcal{L}_{p'}(m)$.
We denote by $A_\gamma$ the subset of $\mathcal{L}_{p'}(n)$ defined as follows:
$$A_\gamma=\{\lambda\in\mathcal{L}_{p'}(n)\ |\ C_{p\6{n_k}}(\lambda)=\gamma\}.$$
Observing that the $p\6{n_k}$-core of a hook partition is a hook partition and using Lemma \ref{lem: last layer} we deduce that $$\mathcal{L}_{p'}(n)=\bigcup_{\gamma\in\mathcal{L}_{p'}(m)}A_\gamma,$$
where the above union is clearly disjoint. 
Given $\gamma\in\mathcal{L}_{p'}(m)$ we notice that the set $A_\gamma$ consists of all those hook partitions of $n$ that are obtained from $\gamma$ by adding $x$-many $p\6{n_k}$-hooks to the first row of $[\gamma]$ and $(a_k-x)$-many to the first column of $[\gamma]$. 
In particular 
$$A_\gamma=\{\lambda_x\ |\ x\in\{0,1,\ldots, a_k\}\},\ \  \text{where}\ \ \lambda_x=(\gamma_1+xp\6{n_k}, 1\6{n-\gamma_1-xp\6{n_k}}).$$ We conclude that $|A_\gamma|=a_k+1$, for all $\gamma\in\mathcal{L}_{p'}(m)$. This, combined to the inductive hypothesis shows that: $$|\mathcal{L}_{p'}(n)|=|\mathcal{L}_{p'}(m)|\cdot(a_k+1)=\big(a_1p^{n_1}\cdot\prod_{j=2}^{k-1}(a_j+1)\big)\cdot(a_k+1)=a_1p^{n_1}\cdot\prod_{j=2}^k(a_j+1).$$

The second statement is a consequence of a very well known fact. Namely that for $\lambda=(n-x,1\6x)\in\mathcal{L}(n)$ we have that $\chi\6\lambda(1)={n-1\choose x}$.
(This can be easily deduced from the hook length formula). 
\end{proof}

Given a natural number $n$ and $S\subset\mathcal{P}(n)$ we let $\mathrm{cd}(S)=\{\chi^\lambda(1)\ |\ \lambda\in S\}$. Moreover, we let $\mathrm{cd}_{p'}^{\mathrm{ext}}(A_n)$ be the set consisting of all the degrees of irreducible characters of $A_n$ of degree coprime to $p$ that extend to $S_n$. Here we find a lower bound to 
$|\mathrm{cd}_{p'}^{\mathrm{ext}}(A_n)|$. 
%This will be shown to be much larger than $3$ in all cases except for a very few specific ones. 

\begin{prop}\label{prop:lowerbound}
Let $n\in\mathbb{N}$. Then $|\mathrm{cd}_{p'}^{\mathrm{ext}}(A_n)|\geq \lfloor |\mathcal{L}_{p'}(n)|/2\rfloor$. 
\end{prop}
\begin{proof}
A hook-partition $\lambda\in\mathcal{L}(n)$ is such that $\lambda=\lambda'$ if and only if $n$ is odd and $\lambda=((n+1)/2, 1^{(n-1)/2})$. Moreover, the set $\mathcal{L}_{p'}(n)$ is clearly closed under conjugation of partitions. It follows that $((n+1)/2, 1^{(n-1)/2})\in\mathcal{L}_{p'}(n)$ if and only if $|\mathcal{L}_{p'}(n)|$ is odd. Let $S=\{\lambda\in\mathcal{L}_{p'}(n)\ |\ \lambda_1>\lambda'_1\}$. The above discussion shows that $|S|=\lfloor |\mathcal{L}_{p'}(n)|/2\rfloor$. The second statement of Proposition \ref{lem:p'hooks} implies that $|\mathrm{cd}(S)|=|S|$. 
We observe that if $\lambda\in S$ then $(\chi^\lambda)_{A_n}$ is irreducible of degree coprime to $p$ and clearly extends to $S_n$. In other words, $(\chi^\lambda)_{A_n}(1)\in\mathrm{cd}_{p'}^{\mathrm{ext}}(A_n)$. 
Moreover, for $\lambda, \mu\in S$ we have that $(\chi\6\lambda)_{A_n}(1)=(\chi\6\mu)_{A_n}(1)$ if and only if $\lambda=\mu$. Thus we conclude that 
$|\mathrm{cd}_{p'}^{\mathrm{ext}}(A_n)|\geq |\mathrm{cd}(S)|=|S|=\lfloor |\mathcal{L}_{p'}(n)|/2\rfloor$, as desired.
\end{proof}

In order to verify that Theorem C holds for alternating groups, 
we aim of to show that $|\mathrm{cd}_{p'}^{\mathrm{ext}}(A_n)|\geq 3$, for all $n\geq 7$. 
Propositions \ref{lem:p'hooks} and \ref{prop:lowerbound} give in most of the cases a much larger lower bound than the one needed. In fact, Proposition \ref{cor: An} below is a consequence of these propositions, together with the analysis of the cases where $n\in\mathbb{N}$ is such that $\lfloor|\mathcal{L}_{p'}(n)|/2\rfloor\leq 2$. 

The following facts are easy applications of the hook-length formula and are important to deal with the few exceptional cases mentioned above.
%\begin{lem}\label{lem:quasihooks}
%Let $n,m,t\in\mathbb{N}$ and $c\in\{2,3\}$ be such that $n=m+t+c$ and such that $0\leq t\leq m-4$. Let $\lambda(t)=(m,c,1^t)$. Then $\chi^{\lambda(t)}(1)<\chi^{\lambda(t+1)}(1)$.
%\end{lem}
\begin{lem}\label{lem:quasihooks}
Let $n,t, c\in\mathbb{N}$ be such that $c\in\{2,3\}$, $n\geq 4+c$ and $0\leq t\leq n-2c$.
Let $\lambda(t)$ be the partition of $n$ defined as $\lambda(t):=(n-c-t,c,1^t)\in\mathcal{P}(n)$. If $0\leq t\leq \lfloor\frac{n-4-c}{2}\rfloor$,
then $$\chi^{\lambda(t)}(1)<\chi^{\lambda(t+1)}(1).$$
\end{lem}
\begin{proof}
To ease the notation we let $m:=\lambda(t)_1=n-c-t$. 
The (strange) hypothesis $0\leq t\leq \lfloor\frac{n-4-c}{2}\rfloor$ is equivalent to say that $\lambda(t+1)_1\geq (\lambda(t+1)')_1$. In turn this is equivalent to say that $0\leq t\leq m-4$.

If $c=2$ then using the hook length formula we observe that
$$\frac{\chi\6{\lambda(t+1)}(1)}{\chi\6{\lambda_{t}}(1)}
=\frac{m(m-2)(t+2)}{(m-1)(t+3)(t+1)}\geq \frac{(m-2)\62}{(m-1)\62}\cdot\frac{m}{m-3}>1,$$
where the first inequality is obtained by replacing $t$ with $m-4$.
The proof of the statement for $c=3$ is completely similar and therefore it is omitted. 
\end{proof}
We care to remark that a more general statement (with arbitrary $c\in\mathbb{N}$) does not hold, as for instance is shown by the pair $(6,5,1,1)$ and $(5,5,1,1,1)$.

\begin{prop}\label{cor: An}
Let $n\geq 7$ be a natural number and let $p>3$ be a prime, then $|\mathrm{cd}_{p'}^{\mathrm{ext}}(A_n)|\geq 3$. 
\end{prop}
\begin{proof}
Let $n=\sum_{j=1}^ka_jp^{n_j}$ be the $p$-adic expansion of $n$, where $0\leq n_1<n_2<\cdots <n_k$ and $1\leq a_j\leq p-1$ for all $j\in\{1,\ldots, k\}$. 
Suppose that 
$$n\notin \{1+p\6{n_2}+p\6{n_3}, 2+p\6{n_2}, 1+a_2p\6{n_2}\ |\ a_2\in\{1,2,3\}\},$$ since $n\geq 7$, then $|\mathrm{cd}_{p'}^{\mathrm{ext}}(A_n)|\geq 3$, by Propositions \ref{lem:p'hooks} and \ref{prop:lowerbound}.
To conclude the proof, we analize the remaining cases one by one. 

Suppose first that $n=1+ap\6k$, for some $a\in\{1,2,3\}$. Let $t\in\mathbb{N}$ be such that $0\leq t\leq \lfloor\frac{n-6}{2}\rfloor$ and let $\lambda(t)=(n-2-t, 2, 1\6t)$. Observe that $\lambda(t)\neq(\lambda(t))'$. By Lemma \ref{lem:quasihooks} we have that $$1<\chi\6{\lambda(0)}(1)<\chi\6{\lambda(1)}(1)<\cdots<\chi\6{\lambda(\lfloor\frac{n-6}{2}\rfloor)}(1)<\chi\6{\lambda(\lfloor\frac{n-6}{2}\rfloor+1)}(1).$$ Moreover, since $h_{(1,1)}(\lambda(t))=ap\6k$ and $C_{p\6k}(\lambda(t))=(1)$, then Lemma \ref{lem: last layer} implies that $\chi\6{\lambda(t)}\in\mathrm{Irr}_{p'}(S_n)$, for all $t\in\{0,1,\dots, \lfloor\frac{n-6}{2}\rfloor\}$. 
Since $n=1+ap\6k\geq 7$ it follows that $n\geq 8$ and hence that $ \lfloor\frac{n-6}{2}\rfloor\geq 1$.  Since $\lambda(t)$ is never equal to the trivial partition,
we conclude that $|\mathrm{cd}_{p'}^{\mathrm{ext}}(A_n)|\geq 3$. 

Let $k,h\in\mathbb{N}$ be such that $k<h$, and suppose that $n=1+p\6k+p\6h$. To ease the notation we let $m=1+p\6k$. By Lemma \ref{lem: last layer} it is easy to observe that $$\mathcal{P}_{p'}(m)=\{(p\6k-t, 2, 1\6{t-1})\ |\ t\in\{1,\ldots, p\6k-2\}\}\cup\{(m), (1\6m)\}.$$
Since $p\geq 5$ we have that $|\mathcal{P}_{p'}(m)|\geq 3$. For each $\gamma\in\mathcal{P}_{p'}(m)$ we let $$\lambda(\gamma):=(\gamma_1+p\6h, \gamma_2,\ldots, \gamma_\ell)\in\mathcal{P}(n).$$
Using Lemma \ref{lem: last layer} it is now routine to check that $\lambda(\gamma)\vdash_{p'}n$. Moreover, since $p\6h>m$ we have that $\lambda(\gamma)\neq(\lambda(\gamma))'$ for all $\gamma\in\mathcal{P}_{p'}(m)$. Therefore, using again Lemma \ref{lem:quasihooks} we conclude that also in this case $|\mathrm{cd}_{p'}^{\mathrm{ext}}(A_n)|\geq 3$. 

Finally suppose that $n=2+p\6k$, for some $k\in\mathbb{N}$. Let $t\in\mathbb{N}$ be such that $0\leq t\leq \lfloor\frac{n-7}{2}\rfloor$ and let $\lambda(t)=(n-3-t, 3, 1\6t)$. Observe that $\lambda(t)\neq(\lambda(t))'$. Since $h_{(1,1)}(\lambda(t))=p\6k$ and $C_{p\6k}(\lambda(t))=(2)$, by Lemma \ref{lem: last layer} we deduce that $\chi\6{\lambda(t)}\in\mathrm{Irr}_{p'}(S_n)$, for all $t\in\{0,1,\dots, \lfloor\frac{n-7}{2}\rfloor\}$. 
If $p\6k\neq 5$ then $n\geq 9$ and hence $\lfloor\frac{n-7}{2}\rfloor\geq 1$. 
Since $\lambda(t)$ is never equal to the trivial partition, using Lemma \ref{lem:quasihooks} we conclude that $|\mathrm{cd}_{p'}^{\mathrm{ext}}(A_n)|\geq 3$.
If $p\6k=5$ then direct verfication shows that $|\mathrm{cd}_{p'}^{\mathrm{ext}}(A_7)|\geq 3$.
\end{proof}

We are now ready to prove Theorem B. 

\begin{cor}
Theorem B holds for all simple non-abelian alternating groups. 
\end{cor}
\begin{proof}
Let $n\geq 5$. 
Since $p>3$ a Sylow $p$-group $P$ of $\mathrm{Aut}(A_n)$ is necessarily contained in $A_n$.  Hence all irreducible characters of $A_n$ are $P$-invariant. With this in mind, we observe that for all $n\geq 7$ Theorem B follows from Proposition \ref{cor: An}.
If $n=5$ then (in accordance with the notation used in the statement of Theorem B) we choose $\alpha=(\chi\6{(4,1)})_{A_5}$ and $\beta$ an irreducible constituent of $(\chi\6{(3,1,1)})_{A_5}$. Similarly for $n=6$ we observe that Theorem B holds by choosing $\alpha=(\chi\6{(4,2)})_{A_6}$ and $\beta$ an irreducible constituent of $(\chi\6{(3,2,1)})_{A_6}$.
\end{proof}

\section{Simple Groups of Lie Type}\label{sec:lietype}

Throughout this section, we will adopt the following notation.  Let $S$ be a simple group of Lie type, by which we mean that there is a simply connected simple linear algebraic group $\bg{G}$ defined over $\overline{\F}_q$ such that $S=G/Z(G)$, where $G:=\bg{G}^F$ is the group of fixed points of $\bg{G}$ under a Steinberg endomorphism $F$.  Here $\overline{\F}_q$ is an algebraic closure of the field $\F_q$ with $q$ elements, and $q$ is a power of some prime.  We further write $G^\ast=(\bg{G}^\ast)^{F^\ast}$, where $(\bg{G}^\ast, F^\ast)$ is dual to $(\bg{G}, F)$.  

We denote by $\iota\colon \bg{G}\hookrightarrow\wt{\bg{G}}$ a regular embedding, as in \cite[Chapter 15]{cabanesenguehard}, and let $\iota^\ast\colon \wt{\bg{G}}^\ast\twoheadrightarrow\bg{G}^\ast$ be the dual surjection of algebraic groups.  When $F$ is a Frobenius endomorphism, %(///needed?///), 
we may extend $F$ to a Frobenius endomorphism on $\wt{\bg{G}}$, which we also denote by $F$, and we let $\wt{G}:=\wt{\bg{G}}^F$ and $\wt{G}^\ast:=(\wt{\bg{G}}^\ast)^{F^\ast}$ be the corresponding group of Lie type and its dual, respectively.   Then $G\lhd \wt{G}$ and the automorphisms of $G$ are generated by the inner automorphisms of $\wt{G}$ (known as inner-diagonal automorphisms of $G$) together with graph and field automorphisms. 

When $\bg{G}$ is of type $A_{n-1}$, we use the notation $\PSL_n^\epsilon(q)$ with $\epsilon\in\{\pm1\}$ to denote $\PSL_n(q)$ for $\epsilon=1$ and $\PSU_n(q)$ for $\epsilon=-1$.   We will also use the corresponding notation for $\SL_n^\epsilon(q), \GL_n^\epsilon(q),$ and $\PGL_n^\epsilon(q)$.

The goal of this section is to prove the following, which is Theorem B for groups of Lie type.
\begin{theorem}\label{thm:mainLie}
Let $S$ be a simple group of Lie type defined over $\F_q$ and let $p>3$ be a prime dividing $|S|$.  Assume that $S$ is not one of $\PSL_2(q), \PSL^\epsilon_3(q),$ $\PSp_4(q)$, or $\tw{2}B_2(q)$. Then there exist two nontrivial characters $\chi_1, \chi_2\in\irrp S$ such that $\chi_1(1)\neq \chi_2(1)$ and both $\chi_1$ and $\chi_2$ extend to $\aut S$.
\end{theorem}

To deal with the exceptions in \prettyref{thm:mainLie}, we prove the following, which is Theorem C.

\begin{theorem}\label{thm:mainLieweak}
Let $S$ be a simple group of Lie type defined over $\F_q$ and let $p>3$ be a prime dividing $|S|$.     Then there exist two nontrivial characters $\chi_1, \chi_2\in\irrp S$ such that $\chi_2(1)\nmid\chi_1(1)$, $\chi_1$ extends to $\aut S$, and $\chi_2$ is invariant under every $p$-group of $\aut S$.
\end{theorem}

\subsection{Defining Characteristic}

We begin by fixing some notation for this section.  Let $q=p^a$ be a power of a prime $p>3$.  We fix $a:=p^b\cdot m$ where $(m,p)=1$, $b\geq 0$, and $m\geq 1$.  In what follows, it will be useful to consider certain elements of $\overline{\F}_q^\times$.  For a positive integer $n$, we will denote by $\zeta_n$ an element of order $p^n-1$ in $\overline{\F}_q^\times$ and by $\xi_n$ an element of order $p^n+1$ in $\overline{\F}_q^\times$.  In particular, $\zeta_1\in\F_p^\times\subseteq\F_q^\times$, $\zeta_m\in\F_q^\times$, and $\xi_m\in \F_{q^2}^\times\setminus\F_q^\times$.  We also have $\xi_1\in\F_{q^2}^\times$ and further $\xi_1\in\F_q^\times$ if and only if $q$ is a square.

%\textcolor{red}{////jkjkjkjkjk don't need!  $p>3$/////}
%\begin{lemma}\label{lem:exclusionsdef}
%\textcolor{red}{Let $S$ be a simple group of Lie type in the following list:
%\[\{B_n(2) (n>2), G_2(2), B_2(2^i), \tw{2}B_2(2^i), G_2(3^i), F_4(2^i), \tw{2}F_4(2^i), \tw{2}G_2(3^i)\}.\]  Then the conclusion of \prettyref{thm:mainLie} holds for $S$ and its defining characteristic.}
%\end{lemma}
%\begin{comment}
%\begin{proof}
%{If $S=B_n(2)$, $G_2(3)$, or $G_2(2)$, using \cite[Theorem 6.8]{malle07}, \cite[Theorems 2.4, 2.5]{Malle08}, and the degrees listed in \cite[Section 13]{carter2}, we see that there are at least two nontrivial unipotent characters in $\irrp S$ of different degree that extend to $\aut S$. 
%
%///$B_2(2^i)\cong Sp_{4}(2^i)$ use thesis?  Here odd degs are all semisimple
%
%///still need $G_2(3^i) (i>1), F_4(2^i), \tw{2}B_2(2^i), \tw{2}F_4(2^i), \tw{2}G_2(3^i)$ (first two $p'$ chars are exactly semi simples.  last three Suz, Ree have only field automs. look at Brunat2009 On the inductive McKay condition in the defining characteristic)
%////////////////Check...is it true? or also exceptons?}
%\end{proof}
%\end{comment}

\subsubsection{Establishing the Basic Strategy}\label{sec:strategy}

Suppose that $F$ is a Frobenius map.    In what follows, we will often proceed following ideas like those in \cite[Theorem 4.5(4)]{SchaefferFrySN2S1} and \cite[Proposition 6.4]{SFTaylorTypeA} to construct characters of $S=G/Z(G)$ satisfying our desired properties, using characters of $\wt{G}$. %  //// remove these citations ??////
  Namely, if $s$ is a semisimple element of $\wt{G}^\ast$, there exists a unique so-called \emph{semisimple} character $\wt{\chi}_s$ of $p'$-degree associated to the $\wt{G}^\ast$-conjugacy class of $s$, and $\wt{\chi}_s(1)=[\wt{G}^\ast\colon C_{\wt{G}^\ast}(s)]_{p'}$.  If further $s\in[\wt{G}^\ast, \wt{G}^\ast]$, then $\wt{\chi}_s$ is trivial on $Z(\wt{G})$, using \cite[Lemma 4.4]{navarrotiep13}.  Furthermore, the number of irreducible constituents of $\chi :=\wt{\chi}_s|_G$ is exactly the number of irreducible characters $\theta \in \Irr(\widetilde{G}/G)$ satisfying $\widetilde{\chi}_s\theta = \widetilde{\chi}_s$, and we have $\irr{\widetilde{G}/G} = \{\widetilde{\chi}_z \mid z \in Z(\widetilde{G}^{\ast})\}$ and $\mathcal{E}(\widetilde{G}, s)\widetilde{\chi}_z = \mathcal{E}(\widetilde{G}, sz)$ for such $z \in Z(\widetilde{G}^{\ast})$, by \cite[13.30]{dignemichel}. Then $\chi$ is irreducible if and only if $s$ is not $\widetilde{G}^\ast$-conjugate to $sz$ for any nontrivial $z\in Z(\widetilde{G}^\ast)$.  Finally, if $\varphi\in\aut{\wt{G}}$ and $\varphi^\ast\colon \wt{G}^\ast\rightarrow\wt{G}^\ast$ is dual to $\varphi$, then \cite[Corollary 2.4]{NavarroTiepTurullCyclo} tells us that $\chi_s^\varphi=\chi_{\varphi^{\ast}(s)}$.

Therefore, in the context of proving \prettyref{thm:mainLie}, we will be interested in showing that there exist two nontrivial semisimple elements $s_1, s_2\in \wt{G}^\ast$ such that:
\begin{itemize}
\item[(1)] $s_1, s_2$ are contained in $[\wt{G}^\ast, \wt{G}^\ast]$;
\item[(2)] for $i=1,2$, $s_i$ is not conjugate to $s_iz$ for any nontrivial $z\in Z(\wt{G}^\ast)$; 
\item[(3)] the $\wt{G}^\ast$-classes of $s_1$ and $s_2$ are $\aut{\wt{G}^\ast}$-invariant; and
\item[(4)]  $|C_{\wt{G}^\ast}(s_1)|_{p'}\neq |C_{\wt{G}^\ast}(s_2)|_{p'}$. 
\end{itemize}

In the context of \prettyref{thm:mainLieweak}, we will need to replace (3) and (4) with:
\begin{itemize}
\item[(3')] the $\wt{G}^\ast$-class of $s_1$ is $\aut{\wt{G}^\ast}$-invariant and that of $s_2$ is fixed by $p$-elements of $\aut{\wt{G}^\ast}$; and
\item[(4')] $|C_{\wt{G}^\ast}(s_1)|_{p'}\nmid |C_{\wt{G}^\ast}(s_2)|_{p'}$. 
\end{itemize}

\subsubsection{The Proofs in Defining Characteristic}
 \begin{proposition}\label{prop:maindef}
Let $S$ be a simple group of Lie type defined in characteristic $p>3$ not in the list of exclusions of \prettyref{thm:mainLie}.  Then the conclusion of \prettyref{thm:mainLie} holds for $S$ and $p$.
\end{proposition}
\begin{proof}
First assume $S$ is one of $G_2(q), F_4(q),$ or $\tw{3}D_4(q)$. The character degrees in these cases are available at \cite{luebeckwebsite}, and the generic character tables are available in CHEVIE \cite{chevie}.  Here $\aut S/S$ is cyclic generated by a field automorphism, so characters extend to $\aut S$ if and only if they are invariant under $\aut S$.  In the case of $G_2(q)$, there is a unique character of degree $(q^4+q^2+1)$ and a unique character of degree $(q^3+\epsilon)$, where $\epsilon=\pm1$ is such that $q\equiv\epsilon\pmod 6$, so the statement holds for $G_2(q)$.  Similarly, $F_4(q)$ has a unique character of degree $(q^8+q^4+1)$ and a unique character of degree $(q^2+1)(q^4+1)(q^8+q^4+1)$.  Since $\tw{3}D_4(q)$ has a unique character of degree $(q^8+q^4+1)$, it suffices in this case to find another member of $\irrp S$  invariant under field automorphisms.  Taking $k$ to be such that $\gamma^k\in\F_p^\times$, where $\gamma$ generates $\F_q^\times$, this is accomplished by $\chi_{13}(k)$ in CHEVIE notation, which has degree $(q+1)(q^8+q^4+1)$. 

Hence we may assume $S$ is not in the above list, and by \prettyref{sec:alt}, we may assume that $S$ is not isomorphic to an alternating group.  Also note that we may assume $G$ does not have an exceptional Schur multiplier. Let $D\leq\aut{\wt{G}}$ be as in \cite[Notation 3.1]{spath12}, so that $D$ is generated by appropriate graph and field automorphisms and $\aut{G}$ is generated by $D$ and the inner automorphisms of $\wt{G}$.  % (////more details here? or up in notation?////) 

%Let $S=G/Z(G)$ for $G$ a group of Lie type of simply connected type, and let $\wt{G}$ and $D$ be as in as in \cite[Notation 3.1]{spath12}, so $G\lhd \wt{G}$ and $\aut(G)$ is generated by $D$ and the inner automorphisms of $\wt{G}$.   (////more details here later////) 

 Since $[\wt{G}:G]$ is coprime to $p$, %and restrictions from $\wt{G}$ to $G$ are multiplicity-free, by \cite{lusztig?}, 
we see by Clifford theory that the set of members of $\irr{\wt{G}}$ lying above members of $\irrp{G}$ is exactly the set $\irrp{\wt{G}}$.  By \cite[Proposition 3.4]{spath12}, for every $\wt{\chi}\in\irrp{\wt{G}}$, there is a character $\chi_0\in\irr{G|\wt{\chi}}$ such that $(\wt{G}\rtimes D)_{\chi_0}=\wt{G}_{\chi_0}\rtimes D_{\chi_0}$ and $\chi_0$ extends to $G\rtimes D_{\chi_0}$.  Further, if $\wt{\chi}|_G=\chi_0$ and $\wt{\chi}$ is $D$-invariant, it further follows that $\chi_0$ extends to $\wt{G}\rtimes D$, following the proof of \cite[Lemma 2.13]{spath12}, since the factor set constructed there for the projective representation of $(\wt{G}\rtimes D)_{\chi_0}$ extending $\chi_0$ is trivial in this case.

Hence it suffices to show that there exist $D$-invariant $\wt{\chi}_1, \wt{\chi}_2\in\irrp{\wt{G}}$ that are trivial on $Z(G)$, have different degrees, and satisfy that $\wt{\chi}_1|_G$ and $\wt{\chi}_2|_G$ are irreducible.  %We proceed following ideas like those in \cite[Theorem 4.5(4)]{SchaefferFrySN2S1} and \cite[Proposition 6.4]{SFTaylorTypeA}. %  //// remove these citations ??////
In particular, it suffices to find $s_1$ and $s_2$ as described in \prettyref{sec:strategy} satisfying conditions (1)-(4).  For condition (2), we note that by \cite[Corollary 2.8(a)]{bonnafe05}, it suffices to show that $C_{\bg{G}^\ast}(\iota^\ast(s_i))$ are connected, and hence by \cite[Exercise 20.16(c)]{MalleTesterman}, to choose $s_i$ such that $(|s_i|, |Z(\bg{G})|)=1$.

First, suppose that $\bg{G}$ is not of type $A_\ell$.  Let $\Phi$ and $\Delta:=\{\alpha_1, \alpha_2, \cdots, \alpha_\ell\}$ be a system of roots and simple roots, respectively, for $\wt{\bg{G}}^\ast$ with respect to a maximal torus $\wt{\bg{T}}^\ast$, following the standard model described in \cite[Remark 1.8.8]{GLS3}. Note that we may assume that $|\Delta|\geq 3$ if $\Phi$ is type $B_\ell$ or $C_\ell$, $|\Delta|\geq 4$ if $\Phi$ is type $D_\ell$, and otherwise $\Phi$ is type $E_6, E_7,$ or $E_8$.  Further, our assumptions imply that if the Dynkin diagram for $\Phi$ has a nontrivial graph automorphism, then all members of $\Delta$ have the same length and that automorphism has order $2$ unless $\Phi$ is of type $D_4$.  

Given $\alpha\in \Phi$, let $h_\alpha$ denote the corresponding coroot, following the notation of \cite{GLS3}.    Notice that for $\alpha\in \Phi$ and $t\in \overline{\F}_q^\times$, we have $h_\alpha(t)\in [\wt{\bg{G}}^\ast, \wt{\bg{G}}^\ast]$.  (See, for example, \cite[Theorem 1.10.1(a)]{GLS3}.)   Let $\delta\in \F_p^\times$ be such that $|\delta|$ is prime to $|Z(\bg{G})|$, if possible.  Otherwise, we have $\bg{G}$ is type $E_7, B_\ell, C_\ell,$ or $D_\ell$ and $p-1$ is a power of $2$.  (Note that if $\bg{G}$ is type $E_6$, and $p-1$ is a power of $3$, then $p=2$, contradicting our assumption that $p>3$).   In these latter cases, let $\delta$ be an element of $\F^\times_{p^2}$ with order prime to $|Z(\bg{G})|$ dividing $p+1$.       

We define $s_1':=h_{\alpha_1}(\delta)h_{\alpha_2}(\delta)\cdots h_{\alpha_{\ell-1}}(\delta)h_{\alpha_\ell}(\delta)$.  Let $\beta$ be a member of $\Delta$ as follows.  For $\Phi$ of type $E_6, E_7,$ or $E_8$, let $\beta:=\alpha_4$.  If $\Phi$ is of type $C_\ell$ or $D_\ell$, let $\beta:=\alpha_\ell$.  If $\Phi$ is type $B_\ell$, let $\beta:=\alpha_1$.
%described below:
%\begin{center}
%\begin{tabular}{|c|c|c|c|c|c|c|}
%\hline
%Type of $\Phi$ & $B_m$ & $C_m$ &  $D_m$ & $E_6$ & $E_7$ & $E_8$ \\ 
%\hline
% $\beta$ & $\alpha_1$ & $\alpha_m$  & $\alpha_m$ & $\alpha_4$& $\alpha_4$ & $\alpha_4$\\
%\hline
%\end{tabular}
%\end{center}
  Let $s_2':=h_{\beta}(\delta)$ for $\Phi$ not of type $D_\ell$; $s_2':=h_{\alpha_\ell}(\delta)h_{\alpha_{\ell-1}}(\delta)$ for $\Phi$ of type $D_\ell$ with $\ell\geq 5$; and $s_2':=h_{\alpha_1}(\delta)h_{\alpha_3}(\delta)h_{\alpha_4}(\delta)$ for $\Phi$ of type $D_4$.  Note then that for $i=1,2$, $s_i'$ is fixed under graph automorphisms and $(|s_i'|, |Z(\bg{G})|)=1$.   

 If $\delta\in\F_q^\times$, we see that the $s_i'$ are $F^\ast$-fixed, and we write $s_i:=s_i'$.  Otherwise, the elements $\alpha_1+\ldots +\alpha_\ell$ and $\beta$ are members of $\Phi$, and in the case of $D_\ell$, we have $\alpha_\ell+\alpha_{\ell-1}=2e_{\ell-1}$ and for $\ell=4$, $\alpha_1+\alpha_3+\alpha_4=e_1-e_2+2e_3$.  In the first case, let $\dot{w}\in N_{\wt{\bg{G}}^\ast}(\wt{\bg{T}}^\ast)$ induce the corresponding reflection in the Weyl group of $\wt{\bg{G}}^\ast$.  In the case of $\alpha_\ell+\alpha_{\ell-1}$ in type $D_\ell$, we may take $\dot{w}$ to be the product of members of $N_{\wt{\bg{G}}^\ast}(\wt{\bg{T}}^\ast)$ inducing the reflections in the Weyl group of $\wt{\bg{G}}^\ast$  for $\alpha_\ell$ and $\alpha_{\ell-1}$, and similarly for $\alpha_1+\alpha_3+\alpha_4$ in the case of $D_4$. In any case, we have $s_i:=s_i'^g$ is $F^\ast$-fixed, where $g\in \wt{\bg{G}}^\ast$ satisfies $g^{-1}F^\ast(g)=\dot{w}$.  (Note that such a $g$ exists by the Lang-Steinberg theorem.)     Hence $s_1$ and $s_2$ are members of $[\wt{G}^\ast, \wt{G}^\ast]$.  
 
 Further, we have constructed $s_1$ and $s_2$ to be fixed under graph automorphisms and such that $|C_{\wt{G}^\ast}(s_1)|_{p'}\neq |C_{\wt{G}^\ast}(s_2)|_{p'}$. The latter can be seen by analyzing the root information in \cite{GLS3} and using the fact that $C_{\wt{\bg{G}}^\ast}(s_i)$ has root system $\Phi_{s_i}$ where $\Phi_{s_i}$ consists of $\alpha\in\Phi$ with $\alpha(s_i)=1$ (see \cite[Proposition 2.3]{dignemichel}). Let $F_p$ denote a generating field automorphism such that $F_p(h_\alpha(t))=h_\alpha(t^p)$ for $\alpha\in\Phi$ and $t\in\overline{\F}_q^\times$.  Then for $i=1,2$, $s_i'$ is $\wt{\bg{G}}^\ast$-conjugate to $F_p(s_i')$, taking for example $\dot{w}$ as the conjugating element when $s_i'\neq F_p(s_i')$.  Hence $s_i$ is also $\wt{\bg{G}}^\ast$-conjugate to $F_p(s_i)$. Since the $C_{\bg{G}^\ast}(s_i)$ are connected, this yields that the $s_i$ are $\wt{G}^\ast$-conjugate to $F_p(s_i)$, using \cite[(3.25)]{dignemichel}.  Then $s_1$ and $s_2$ are semisimple elements satisfying (1)-(4), as desired.   
 
 Now let $\Phi$ be of type $A_{n-1}$, so that $\wt{{G}}\cong\wt{{G}}^\ast\cong \GL^\epsilon_n(q)$, $G\cong [\wt{G}^\ast, \wt{G}^\ast]=\SL_n^\epsilon(q)$, and $G^\ast\cong \PGL_n^\epsilon(q)$, with $\epsilon\in\{\pm1\}$ and $n\geq 4$.   If $(p, \epsilon)\neq (5, +1)$, let $\delta\in C_{p-\epsilon}$, viewed as a subgroup of $\F_{q^2}$, be such that $\delta=\zeta_1$ in case $\epsilon=1$ and $\delta=\xi_1$ in case $\epsilon=-1$. Note that $|\delta|>4$, from the conditions on $p$.  If $(p, \epsilon)=(5, +1)$, let $\delta$ be an element of order $6$ in $\F_{25}$.  Then in any case, we have $\delta$ chosen such that $|\delta|>4$. 
 
 Recall that the class of a semisimple element in $\GL_n^\epsilon(q)$ is determine by its eigenvalues.  Let $s_1\in \wt{G}^\ast$ have eigenvalues $\{\delta, \delta^{-1}, 1, 1, \ldots, 1\}$ and $s_2$ have eigenvalues $\{\delta, \delta, \delta^{-1}, \delta^{-1}, 1, \ldots, 1\}$.  Then $C_{\wt{G}^\ast}(s_1)\cong \GL^\epsilon_{n-2}(q)\times \GL_1^\epsilon(q)^2$ and $C_{\wt{G}^\ast}(s_2)\cong \GL_{n-4}^\epsilon(q)\times \GL_2^\epsilon(q)^2$, unless $\epsilon=-1$ and $q$ is square or $(p, \epsilon)=(5, +)$ and $q$ is nonsquare, in which case $C_{\wt{G}^\ast}(s_1)\cong \GL^\epsilon_{n-2}(q)\times \GL_1(q^2)$ and $C_{\wt{G}^\ast}(s_2)\cong \GL^\epsilon_{n-4}(q)\times \GL_2(q^2)$.  In any case, we therefore have $|C_{\wt{G}^\ast}(s_1)|_{p'}\neq |C_{\wt{G}^\ast}(s_2)|_{p'}$.     Since the graph automorphism acts via inverse-transpose on $\GL_n(q)$ and the generating field automorphism acts on semisimple elements by raising the eigenvalues to the power of $p$, we see that the $\wt{G}^\ast$-classes of $s_1$ and $s_2$ are each $\aut{\wt{G}^\ast}$-invariant.   Further, as $s_1$ and $s_2$ have determinant $1$, they are contained in $[\wt{G}^\ast, \wt{G}^\ast]=\SL_n^\epsilon(q)$.
 
 Now, in this case, $Z(\wt{G}^\ast)$ is comprised of matrices of the form $\mu\cdot I_n$, where $\mu$ is an element of $C_{q-\epsilon}$, viewed as a subgroup of $\F_{q^2}$.  Since $|\delta|\neq 2$, we see by comparing eigenvalues that $s_1$ cannot be conjugate to $s_1 z$ for any nontrivial $z\in Z(\wt{G}^\ast)$.  Now, if $s_2z$ is conjugate to some $z\in Z(\wt{G}^\ast)$, then there is some $\mu$ as above such that $\delta\mu=\delta^{-1}$ and $\delta^{-1}\mu=\delta$, except possibly if $n=6$, in which case $\delta\mu=1$ is another possibility.  In the latter case, $\mu=\delta^{-1}$, so $s_2z$ has eigenvalues $\{1, 1, \delta^{-2}, \delta^{-2}, \delta^{-1}, \delta^{-1}\}$, so we must have $\delta^{-2}=\delta$, contradicting $|\delta|>4$.  Hence we are in the case $\delta\mu=\delta^{-1}$ and $\delta^{-1}\mu=\delta$.  Then $\mu=\delta^{-2}=\delta^2$, again contradicting $|\delta|>4$. Hence in all cases, we have exhibited the desired elements $s_1, s_2\in\wt{G}^\ast$, and the proof is complete.  
%////then type $A$ with $m\geq2$? using eigenvalue args: $(\delta, \delta, \delta^{-2})$ and $(\delta, \delta^{-1}, 1)$?? no not graph inv...I think $PSL_3(q)$ needs to be added to exception below. also $PSL_4(q)$ when $p=3$? jk $p$ is larger than 3!///
\end{proof}

We remark that the conclusion of \prettyref{thm:mainLie} fails for $\PSL_2(q)$ and $\PSL_3(q)$ when $q$ is a square. Indeed, in these cases, the only option for $\aut{\wt{G}^\ast}$-invariant semisimple classes of $\wt{G}^\ast$ would be comprised of elements with eigenvalues $\{\delta, \delta^{-1}\}$ and $\{\delta, \delta^{-1}, \pm1\}$, respectively, where $\delta^p=\delta$ or $\delta^{-1}$.  If $q$ is a square, both options would yield $\delta\in\F_q$, so all corresponding semisimple characters have the same degree.  %Similarly, in the case of $PSU_3(q)$, invariant semisimple classes in $[\wt{G}^\ast, \wt{G}^\ast]$ would have eigenvalues $\{\delta, \delta^p, 1\}$ with $\delta^{p+1}=1$ or $\{\delta, \delta^p, \delta^{p^2}\}$ with $\delta^{p^2+p+1}=1$, which do not exist when $q$ is square, respectively a power of $3$.
  However, we can show the following:

\begin{lemma}\label{lem:PSL23def}
Let $S\cong \PSL_2(q)$, $\PSL^\epsilon_3(q)$, or $\PSp_4(q)$ with $q$ a power of a prime $p>3$.  Then the conclusion of \prettyref{thm:mainLieweak} holds for $S$ and $p$.  %Then there exist $\chi_1, \chi_2\in\irrp S$ such that $\chi_1$ extends to $\aut S$, $\chi_2$ is invariant under every $p$-subgroup of $\aut S$, and $\chi_2(1)\nmid\chi_1(1)$.  
Moreover, if $q$ is not square and further $p\neq 5$ in the case of $\PSL_2(q)$, then the conclusion of \prettyref{thm:mainLie} holds for $S$ and $p$.
\end{lemma}

\begin{proof}
Note that since $p>3$, the graph and diagonal automorphisms are not in a $p$-subgroup of $\aut S$.   In particular, the only $p$-elements of $D$ are induced from field automorphisms of $p$-power order.  Arguing as in the proof of \prettyref{prop:maindef}, it suffices to show that there are two semisimple elements $s_1, s_2$ of $\wt{G}^\ast$ satisfying conditions (1), (2), (3'), and (4') from \prettyref{sec:strategy}.  Throughout the proof, let $q=p^a$ where $a=p^bm$ with $(m,p)=1$.

(i)   First suppose that $p>5$ and $S\cong \PSL_2(q)$.  Note that since $p>5$, both $\zeta_1$ and $\xi_m$ have order larger than $4$.  Let $s_1\in\wt{G}^\ast$ have eigenvalues $\{\zeta_1, \zeta_1^{-1}\}$, and $s_2$ have eigenvalues $\{\xi_m, \xi_m^{-1}\}$.  Then $s_1$ and $s_2$ satisfy conditions (1), (2), (3'), and (4') from \prettyref{sec:strategy}.   Here $C_{\wt{G}^\ast}(s_1)\cong \GL_1(q)^2$, $C_{\wt{G}^\ast}(s_2)\cong \GL_1(q^2)$, and the corresponding semisimple characters have degrees $\chi_1(1)=q+1$ and $\chi_2(1)=q-1$, respectively. This proves the first statement when $p>5$ and $S\cong \PSL_2(q)$.

(ii) Next let $p=5$ and $S\cong \PSL_2(q)$.   First suppose $m>1$.  Here we may take $s_1\in \wt{G}^\ast$ to have eigenvalues $\{\xi_1, \xi_1^{-1}\}$, and take $s_2$ to have eigenvalues $\{\zeta_m, \zeta_m^{-1}\}$ if $2\nmid m$ and eigenvalues $\{\xi_m, \xi_m^{-1}\}$ if $2|m$.  Then if $2\nmid m$, we have $C_{\wt{G}^\ast}(s_1)\cong \GL_1(q^2)$ and $C_{\wt{G}^\ast}(s_2)\cong \GL_1(q)^2$.  When $2\mid m$, the centralizers of $s_1$ and $s_2$ are reversed.  In either case, however, we obtain $\chi_1$ and $\chi_2$ satisfying the required conditions, as above. 
% (iii) Now let $p=5$ and $S\cong PSL_2(q)$.   First let $m>1$.  Here we may take $s_1\in \wt{G}^\ast$ to have eigenvalues $\{\lambda, \lambda^5\}$, with $|\lambda|=6$ and take $s_2$ to have eigenvalues $\{\lambda, \lambda^{5^m}\}$, with $|\lambda|=p^m-1$ if $2\nmid m$ and $p^m+1$ if $2|m$.  Then if $2\nmid m$, we have $C_{\wt{G}^\ast}(s_1)\cong GL_1(q^2)$ and $C_{\wt{G}^\ast}(s_2)\cong GL_1(q)^2$.  When $2\mid m$, the centralizers of $s_1$ and $s_2$ are reversed.  In either case, however, we have constructed $\chi_1$ and $\chi_2$ satisfying the required conditions, as above.  

 Now let $m=1$, so $q=5^{5^b}$.  Then $q\equiv 1\pmod 4$ and $q\equiv-1\pmod3$.  We obtain an $\aut S$-invariant character $\chi_1$ constructed using $s_1$ as in the case in the previous paragraph. Here $\chi_1(1)=q-1$. Further, we see from the generic character table that there is a character of degree $(q+1)/2$ which is invariant under every field automorphism, completing the proof of the first statement for $PSL_2(q)$.

(iii) Now let $S$ be $\PSL^\epsilon_3(q)$, with $p\geq 5$.  Notice that an element $s\in \wt{G}^\ast$ with  eigenvalues $\{\delta, \delta^{-1}, 1\}$ with $|\delta|\geq 4$ cannot be conjugate to $sz$ for any nontrivial $z\in Z(\wt{G}^\ast)$, as then the sets $\{\delta, \delta^{-1}, 1\}$ and $\{\delta\mu, \delta^{-1}\mu, \mu\}$ are the same for some $\mu\in C_{q-\epsilon}$, yielding $\delta^3=1$.   Note $|\zeta_1|\geq4$ since $p\geq 5$.  Hence we may construct $s_1$ and $s_2$ analogously to case (i) above.  Namely, let $s_1$ have eigenvalues $\{\zeta_1, \zeta_1^{-1}, 1\}$, and let $s_2$ have eigenvalues $\{\xi_m, \xi_m^{-1}, 1\}$.   %Again we see the corresponding semisimple characters of $\wt{G}$ restrict irreducibly to $G$ and are trivial on $Z(\wt{G})$.  The semisimple character corresponding to $s_1$  is invariant under $\aut S$, and that corresponding to $s_2$ is invariant under every $p$-element of $D$.  This yields the desired characters $\chi_1, \chi_2 \in\irrp S=\irr_{p'}(G/Z(G))$.     
Then $s_1$ and $s_2$ satisfy properties (1),(2),(3'),(4') of \prettyref{sec:strategy}.  Here $C_{\wt{G}^\ast}(s_1)\cong \GL_1(q)^3$ in case $\epsilon=1$, and $C_{\wt{G}^\ast}(s_1)\cong \GL_1(q^2)\times \GU_1(q)$ in case $\epsilon=-1$.  Further, $C_{\wt{G}^\ast}(s_2)\cong \GL_1(q^2)\times \GL_1(q)$ in case $\epsilon=1$ and $C_{\wt{G}^\ast}(s_2)\cong \GU_1(q)^3$ in case $\epsilon=-1$.   Hence $\chi_1(1)=(q+1)(q^2+\epsilon q+1)$ and $\chi_2(1)=(q-1)(q^2+\epsilon q+1)$.   %(iii) Now let $S$ be $PSL^\pm_3(q)$, with $p\geq 5$.  Notice that an element $s\in \wt{G}^\ast$ with  eigenvalues $\{\delta, \delta^{-1}, 1\}$ with $|\delta|\geq 4$ cannot be conjugate to $sz$ for any nontrivial $z\in Z(\wt{G}^\ast)$, as then the sets $\{\delta, \delta^{-1}, 1\}$ and $\{\delta\mu, \delta^{-1}\mu, \mu\}$ are the same for some $\mu\in GL^\pm_1(q)$, yielding $\delta^3=1$.   Hence we may construct $s_1$ and $s_2$ analogously to the case $PSL_2(q)$ and $p>5$ above.  Namely, let $s_1$ have eigenvalues $\{\lambda, \lambda^{-1}, 1\}$ where $\lambda^p=\lambda$ and $|\lambda|>3$, which is possible since $p\geq 5$, and let $s_2$ have eigenvalues $\{\lambda, \lambda^{p^m}, 1\}$, where $\lambda\in\F_{q^2}^\times\setminus\F_q^\times$ has order $p^m+1$.   Again we see the corresponding semisimple characters of $\wt{G}$ restrict irreducibly to $G$ and are trivial on $Z(\wt{G})$.  The semisimple character corresponding to $s_1$  is invariant under $\aut S$, and that corresponding to $s_2$ is invariant under every $p$-element of $D$.  This yields the desired characters $\chi_1, \chi_2 \in\irrp S=\irr_{p'}(G/Z(G))$.     Here $C_{\wt{G}^\ast}(s_1)\cong GL_1(q)^3$ in case $S\cong PSL_3(q)$, and $C_{\wt{G}^\ast}(s_1)\cong GL_1(q^2)\times GU_1(q)$ in case $S\cong PSU_3(q)$.  Further, $C_{\wt{G}^\ast}(s_2)\cong GL_1(q^2)\times GL_1(q)$ in case $S\cong PSL_3(q)$ and $C_{\wt{G}^\ast}(s_2)\cong GU_1(q)^3$ in case $S\cong PSU_3(q)$.   Hence $\chi_1(1)=(q+1)(q^2\pm q+1)$ and $\chi_2(1)=(q-1)(q^2\pm q+1)$.   This completes the proof of the first statement.

(iv) Now let $S=\PSp_4(q)$.  The character table of $G=\Sp_4(q)$ and of $\wt{G}=\CSp_4(q)$ are available in \cite{srinivasan1968} and \cite{breeding}, respectively.  From this we see that the characters in the families $\chi_8(k)$ and $\chi_6(\ell)$ with $k\in\{1,...,(q-3)/2\}, \ell\in\{1,....,(q-1)/2\}$ in the notation of \cite{srinivasan1968}, with degrees $(q+1)(q^2+1)$ and $(q-1)(q^2+1)$, respectively, contain $Z(G)$ in the kernel and extend to $\wt{G}$.  Further, comparing notations shows these extensions are invariant under the same field automorphisms as the characters of $G$. Choosing $k$ and $\ell$ such that $\gamma^k=\zeta_1$ and $\eta^\ell=\xi_m$, where $\gamma$ generates $\F_q^\times$ and $\eta$ generates the cyclic group of size $q+1$ in $\F_{q^2}^\times$, we see that $\chi_1:=\chi_8(k)$ and $\chi_2:=\chi_6(\ell)$ satisfy the desired properties.  

(v) Finally, let $q$ be nonsquare, and further assume $p>5$ in the case $S\cong \PSL_2(q)$.   Let $s_1$ be as in part (i) in the case $\PSL_2(q)$ and as in part (iii) for $\PSL_3^\epsilon(q)$.  Let $s_2$ in $\wt{G}^\ast$ have eigenvalues $\{\xi_1, \xi_1^{-1}\}$ or $\{\xi_1, \xi_1^{-1}, 1\}$, respectively. In this case $\xi_1\in\F_{q^2}^\times\setminus\F_q^\times$ has order larger than $4$. These elements satisfy conditions (1)-(4) discussed in \prettyref{sec:strategy}.  In particular, when $S\cong \PSL_2(q)$, the semisimple characters corresponding to $s_1$ and $s_2$ have degrees $q+1$ and $q-1$, respectively.  When $S\cong \PSL^\epsilon_3(q)$, the semisimple characters corresponding to $s_1$ and $s_2$ have degrees $(q+1)(q^2+\epsilon q+1)$ and $(q-1)(q^2+\epsilon q+1)$, respectively.  Then in these cases, as in the proof of \prettyref{prop:maindef}, the conclusion of \prettyref{thm:mainLie} holds.
%(v) Finally, let $q$ be nonsquare, and assume $p>5$ in the case $S\cong PSL_2(q)$.   Let $s_1$ be as in part (i) in the case $PSL_2(q)$ and as in part (iii) for $PSL_3^\pm(q)$.  Let $s_2$ in $\wt{G}^\ast$ have eigenvalues $\{\lambda, \lambda^p\}$ or $\{\lambda, \lambda^p, 1\}$, respectively, where $\lambda\in\F_{q^2}^\times\setminus\F_q^\times$ has order $p+1$, which is necessarily larger than $4$.  In this case, these elements satisfy the conditions discussed in the proof of \prettyref{prop:maindef}.  In particular, when $S\cong PSL_2(q)$, the semisimple characters corresponding to $s_1$ and $s_2$ have degrees $q+1$ and $q-1$, respectively.  When $S\cong PSL^\pm_3(q)$, the semisimple characters corresponding to $s_1$ and $s_2$ have degrees $(q+1)(q^2\pm q+1)$ and $(q-1)(q^2\pm q+1)$, respectively.  Then in these cases, as in the proof of \prettyref{prop:maindef}, the conclusion of \prettyref{thm:mainLie} holds.

Now let $S=\PSp_4(q)$.  Here we let $\chi_1$ be as in (iv), and let $\chi_2$ be the character $\chi_6(\ell)$, where $\ell$ is now chosen so that $\eta^\ell=\xi_1$ is a $p+1$ root of unity in $\F_{q^2}\setminus\F_q$.  Then $\chi_1$ and $\chi_2$ both extend to $\aut S$ and have different degrees, completing the proof.
\end{proof}

\subsection{Non-Defining Characteristic}
 In this section, we address the proofs of Theorems \ref{thm:mainLie} and \ref{thm:mainLieweak} in the case $p\nmid q$.
 
 \begin{proposition}\label{prop:mainnondef}
  Let $S$ be a simple group of Lie type defined over $\F_q$ not in the list of exclusions of \prettyref{thm:mainLie}, and let $p>3$ be a prime dividing $|S|$ but not dividing $q$.  Or let $S=\tw{2}B_2(q^2)$ where $q^2:=2^{2m+1}$ and let $p>3$ be a prime dividing $|S|$ but not dividing $q^2-1$.   Then the conclusion of \prettyref{thm:mainLie} holds for $S$ and $p$.
\end{proposition}
\begin{proof}
By \cite[Theorem 2.4]{Malle08}, every unipotent character of $S$ extends to its inertia group in $\aut S$, so it suffices to find two unipotent characters with different degrees in $\irrp S$ that are invariant under $\aut S$.  Since the Steinberg character $\mathrm{St}_S$ is one such character, we aim to exhibit another nontrivial unipotent character of $p'$-degree invariant under $\aut S$.  

Further, \cite[Theorem 2.5]{Malle08}, yields that every unipotent character of $S$ is invariant under $\aut S$ unless $S$ is a specifically stated exception for one of $D_n(q)$ with $n$ even, $B_2(q)$ with $q$ even, $G_2(q)$ with $q$ a power of $3$, or $F_4(q)$ with $q$ even.  

The unipotent characters of classical groups are indexed by partitions in case of type $A_{n-1}$ and $\tw{2}A_{n-1}$ and by ``symbols" in the other types.  Discussions of these symbols and the corresponding character degrees are available in \cite[Section 13.8]{carter2}. In \prettyref{tab:unipclassical}, we list two unipotent characters for each classical type that extend to $\aut S$ by \cite[Theorems 2.4 and 2.5]{Malle08}.  Further, $p>3$ cannot divide the degree of both characters listed simultaneously, which can be seen, for example, by an application of \cite[Lemma 5.2]{malle07}.  Hence taking $\chi_1=\mathrm{St}_S$ and $\chi_2$ the character listed whose degree is not divisible by $p$, the desired statement holds in the case of classical types.

\begin{table}\begin{center}
\caption{}\label{tab:unipclassical}
\small
\begin{tabular}{|c|c|c|}\hline
Type & Partition/Symbol indexing $\chi$ & $\chi(1)_{q'}$\\
\hline
$A_{n-1}, n\geq 4$ & $(1, n-1)$ & ${\frac{q^{n-1}-1}{q-1}}$ \\

& $(2, n-2)$ & ${\frac{(q^n-1)(q^{n-3}-1)}{(q-1)(q^2-1)}}$\\
\hline
$\tw{2}A_{n-1}, n\geq 4$ &  $(1, n-1)$ & ${\frac{q^{n-1}-(-1)^{n-1}}{q+1}}$ \\ 

&  $(2, n-2)$ &  ${\frac{(q^n-(-1)^{n})(q^{n-3}-(-1)^{n-3})}{(q+1)(q^2-1)}}$\\
\hline
$B_n$ or $C_n$, $n\geq 3$ or $n=2$ and $q$ odd &  $1 \hbox{ } n \choose 0$ & ${\frac{(q^{n-1}-1)(q^n+1)}{2(q-1)}}$ \\

&  $0 \hbox{ } n \choose 1$ & ${\frac{(q^{n-1}+1)(q^n-1)}{2(q-1)}}$\\
\hline
$B_2$ with $q$ a power of $2$ &  $0 \hbox{ } 1 \hbox{ } 2 \choose -$ & $(q-1)^2/2$ \\ 

&  $0 \hbox{ } 2 \choose 1$ &  $(q+1)^2/2$ \\
\hline
  $D_n, n\geq 5$ & $n-1 \choose 1$ & ${\frac{(q^{n}-1)(q^{n-2}+1)}{q^2-1}}$ \\ 

  &  $1 \hbox{ } n \choose 0 \hbox{ } 1$ & ${\frac{(q^{n-1}+1)(q^{n-1}-1)}{q^2-1}}$ \\
  \hline
  $D_4$& $1 \hbox{ } 3 \choose 0 \hbox{ } 2$ & ${\frac{(q+1)^3(q^3+1)}{2}}$ \\ 

  &  $1 \hbox{ } 2 \choose 0 \hbox{ } 3$ &  ${\frac{(q^2+1)^2(q^2+q+1)}{2}}$ \\
  \hline
$\tw{2}D_n, n\geq4$ & $1 \hbox{ } n-1 \choose - $ & ${\frac{(q^{n}+1)(q^{n-2}-1)}{q^2-1}}$ \\ 

&  $0 \hbox{ } 1 \hbox{ } n\choose 1$ &  ${\frac{(q^{n-1}+1)(q^{n-1}-1)}{q^2-1}}$\\
\hline
\end{tabular}\end{center}
\end{table}
\normalsize

Similarly, for groups of exceptional type, Suzuki and Ree groups, and $\tw{3}D_4(q)$, by observing the explicit list of unipotent character degrees in \cite[Section 13.9]{carter2}, we see that there is likewise always a second unipotent character with degree not divisible by $p$, except in the case of $\tw{2}B_2(q^2)$ or $\tw{2}G_2(q^2)$ with $p|(q^2-1)$.  Further, these can again be chosen not to be one of the exceptions listed in \cite[Theorem 2.5]{Malle08}.  When $S=\tw{2}G_2(q^2)$ and $p\mid(q^2-1)$, we may consider instead the unique character of degree $q^4-q^2+1$, which must be invariant under $\aut S$, and hence extends since $\aut S/S$ is cyclic.  
%Let $G$ be a group of Lie type of simply connected type such that $S=G/Z(G)$.  Let $e$ be the multiplicative order of $q$ modulo $p$ if $p$ is odd and modulo $4$ if $p=2$.  Let $\bf{S}_e$ be a Sylow $\Phi_e$-torus of $G$ and $L:=C_{G}(\bf{S}_e)$.  By \cite[Corollary 6.6 and Proposition 7.3]{malle07}, unipotent characters of $G$ of $p'$-degree are in bijection with pairs $(\la, \phi)$, up to $W_{G}(L)$-conjugacy, where $\la \in\irr_{p'}(L)$ is a unipotent character and $\phi\in\irr_{p'}(W_{G}(L,\la))$.  Note that the relative Weyl group $W_G(L,\la)$ is a finite complex reflection group.
%///argue that we can always find a pair $(\la, \phi)$ such that the corresponding character is neither $1_G$ nor $\mathrm{St}_G$, and not the cases excluded in \cite[Theorem 2.5]{Malle08}///  e.g. choose $(1, \phi)$ with $\phi\neq 1$? really suffices to argue that there are at least 3 distinct pairs, so one must be nontrivial, nonsteinberg
 \end{proof}

Together, note that Propositions \ref{prop:maindef} and \ref{prop:mainnondef} prove \prettyref{thm:mainLie}.
%We remark that \cite[Theorem 6.8]{malle07} implies that except in a small number of cases, there is exactly one unipotent character of degree prime to $p$ when $p$ is the defining characteristic of $S$, and hence the strategy used in \prettyref{prop:mainnondef} above will not work in general when $p|q$.
We also note the following, which follows from the proof of \prettyref{lem:PSL23def} and ideas from \prettyref{prop:mainnondef}, together with the fact that $\PSL_3^\epsilon(q)$ has a unipotent character of degree $q(q+\epsilon)$.
\begin{lemma}\label{lem:PSL23}
Let $p>3$ be a prime and let $S\cong \PSL_2(q)$ or $\PSL^\epsilon_3(q)$ be simple with $q$ a power of a prime $r>3$, where $r\neq p$. Then the conclusion of \prettyref{thm:mainLieweak} holds for $S$ and $p$.  % Then there exist $\chi_1, \chi_2\in\irrp S$ such that $\chi_1$ extends to $\aut S$, $\chi_2$ is invariant under every $p$-subgroup of $\aut S$, and $\chi_2(1)\nmid\chi_1(1)$.  
Moreover, the conclusion of \prettyref{thm:mainLie} holds for $S$ and $p$ in the following situations:
\begin{itemize}

\item $S=\PSL_2(q)$, $r>5$, and $p\nmid(q+1)$ or $q$ is not square;
%\item $S=PSL_2(q)$, $r=5$, and /////conditions!!!/////
\item $S=\PSL_3(q)$ and $p\nmid (q+1)$ or $q$ is not square;
\item $S=\PSU_3(q)$.
\end{itemize}
\end{lemma}

To complete the proof of \prettyref{thm:mainLieweak}, we need to consider $\tw{2}B_2(2^{2n+1})$, $\PSL_2(q)$, and $\PSL_3^\epsilon(q)$ when $q$ is a power of $2$ or $3$ and $p\nmid q$. These are treated in the next two Lemmas.

\begin{lemma}\label{lem:SL2323}
Let $p>3$ be a prime and let $S\cong \PSL_2(q)$ or $\PSL^\epsilon_3(q)$ be simple with $q$ a power of $2$ or $3$.  Then the conclusion of \prettyref{thm:mainLieweak} holds for $S$ and $p$.  %Then there exist $\chi_1, \chi_2\in\irrp S$ such that $\chi_1$ extends to $\aut S$, $\chi_2$ is invariant under every $p$-subgroup of $\aut S$, and $\chi_2(1)\nmid\chi_1(1)$.  
%/////Are there nice conditions in this case for 1.1 holding???///// possibly combine with 1.8////
\end{lemma}
\begin{proof}
Let $r\in\{2, 3\}$.  Note that we omit the cases $\PSL_2(r)$, since $S$ is simple.  As before, we take $\chi_1=\mathrm{St}_S$, and show that there exists $\chi_2$ satisfying the desired properties.

First suppose that $q=r^{p^b}$ for some positive integer $b$.  Notice then that $p\nmid (q^2-1)$.  Indeed, otherwise we have $p|\Phi_{2p^c}(r)$ or $p|\Phi_{p^c}(r)$ for some nonnegative integer $c$, and hence by \cite[Lemma 5.2]{malle07}, we have $r$ has order $1, 2, $ or a power of $p$ modulo $p$.  Since the latter is impossible, it follows that $p\mid (r^2-1)$, which is impossible since $p>3$ and $r\leq 3$.  If $S=\PSL^\epsilon_3(q)$, we may take $\chi_2$ to be the unipotent character of degree $q^2+\epsilon q$, and in fact the conclusion of \prettyref{thm:mainLie} holds in this case using  \cite[Theorems 2.4 \& 2.5]{Malle08}.   In the case $S=\PSL_2(q)$ and $r=3$, we have $q\equiv -1\pmod 4$, and we may take $\chi_2$ to be a character of degree $(q-1)/2$, which is fixed by the field automorphisms.   In the case $r=2$ and $S\cong \PSL_2(q)$, let $s\in \wt{G}^\ast$ have eigenvalues $\{\delta, \delta^{-1}\}$ where $|\delta|=3$.  Then the corresponding semisimple character of $\wt{G}$ restricts to $\SL_2(q)\cong \PSL_2(q)$ irreducibly, is fixed by field automorphisms, and has degree $q\pm1$.  Letting $\chi_2$ be the corresponding irreducible character of $S$, we are done in this case.

So we may assume $q=r^a$ where $r\in\{2,3\}$, and $a=p^bm$ with $m>1$ and $(m,p)=1$.   Let $S=\PSL^\epsilon_3(q)$.  If $p\nmid(q+\epsilon)$, we may again take $\chi_2$ to be the unipotent character of degree $q^2+\epsilon q$, and the conclusion of \prettyref{thm:mainLie} holds.  So assume that $p\mid(q+\epsilon)$.  Arguing similarly to above, we see that this excludes the case $(r,m)=(2,2)$ if $\epsilon=-1$, so we may further assume $(r,m)\neq (2,2)$ if $\epsilon=-1$.  Taking $s\in\wt{G}^\ast$ to have eigenvalues $\{\delta, \delta^{-1}, 1\}$ with $|\delta|=r^m+\epsilon$, we may argue as before to obtain a character $\chi_2$ of $S$ with degree $(q-\epsilon)(q^2+\epsilon q+1)$, which is not divisible by $p$, that is invariant under every $p$-element of $\aut S$. 

Now let $S=\PSL_2(q)$.  Then taking $s_i\in\wt{G}^\ast$ for $i=1,2$ to have eigenvalues $\{\delta_i, \delta_i^{-1}\}$, with $|\delta_1|=r^m-1$ and $|\delta_2|=r^m+1$, we obtain characters with degree $(q+1)$ and $(q-1)$ of $S$ that are invariant under $p$-elements of $\aut S$. Since $p>3$ cannot divide both of these character degrees, we may make an appropriate choice for $\chi_2$, which completes the proof.
\end{proof}

\begin{lemma}\label{lem:suz}
Let $S$ be a simple Suzuki group $\tw{2}B_2(q^2)$ with $q^2=2^{2n+1}$ and let $p>3$ be a prime dividing $q^2-1$.  Then the conclusion of \prettyref{thm:mainLieweak} holds for $S$ and $p$.  %Then there exist $\chi_1, \chi_2\in\irrp S$ such that $\chi_1$ extends to $\aut S$, $\chi_2$ is invariant under every $p$-subgroup of $\aut S$, and $\chi_2(1)\nmid\chi_1(1)$.
\end{lemma}
\begin{proof}
As before, we may take $\chi_1$ to be the Steinberg character.  Now, since $\aut S/S$ is cyclic of size $2n+1$, generated by field automorphisms, it suffices to exhibit a character $\chi_2$ with degree coprime to $p$ that is invariant under field automorphisms of $p$-power order.  Let $2n+1=p^bm$ with $(m,p)=1$.  Arguing as in \prettyref{lem:SL2323}, we see $m>1$, since $p|(q^2-1)$.  Hence letting $s$ be such that $\gamma^s$ has order $2^m-1$, where $\gamma$ has order $q^2-1$, we may take $\chi_2$ to be the character $\chi_5(s)$ in CHEVIE notation.  Then $\chi_2$ has degree $q^4+1$ and is invariant under $p$-elements of $\aut S$.
%///need different consideration for  $\tw{2}B_2$, $p|(q^2-1)=2^{2n+1}-1$ - have to use $\chi_5(1)=q^4+1$ somehow in CHEVIE notation - values of interest $\zeta_2+\zeta_2^{-1}$ where $\zeta_2$ is $(q^2-1)$ root of 1. //////
\end{proof}

\prettyref{thm:mainLieweak} now follows by combining Lemmas \ref{lem:PSL23def}, \ref{lem:PSL23}, \ref{lem:SL2323}, and \ref{lem:suz} with \prettyref{thm:mainLie}, which completes the proofs of Theorems A-C.

%%%%%%%
\bibliographystyle{plain}
\bibliography{researchreferences}

\end{document}